\documentclass{mcom-l}

\usepackage{array}
\usepackage{mathtools}
\usepackage[utf8]{inputenc}
\usepackage{amsthm}
\usepackage{amssymb}

\def\house#1{{%
    \setbox0=\hbox{$#1$}
    \vrule height \dimexpr\ht0+1.4pt width .4pt depth \dp0\relax
    \vrule height \dimexpr\ht0+1.4pt width \dimexpr\wd0+2pt depth \dimexpr-\ht0-1pt\relax
    \llap{$#1$\kern1pt}
    \vrule height \dimexpr\ht0+1.4pt width .4pt depth \dp0\relax
}}

\def\roof#1{{%
    \setbox0=\hbox{$#1$}
    \vrule height \dimexpr\ht0+1.4pt width .8pt depth \dp0\relax
    \vrule height \dimexpr\ht0+1.8pt width \dimexpr\wd0+2pt depth \dimexpr-\ht0-1pt\relax
    \llap{$#1$\kern1pt}
    \vrule height \dimexpr\ht0+1.4pt width .8pt depth \dp0\relax
}}
\newtheorem{theorem}{Theorem}[section]
\newtheorem{lemma}[theorem]{Lemma}
\newtheorem{conjecture}[theorem]{Conjecture}
\newtheorem{corollary}[theorem]{Corollary}

\theoremstyle{definition}

\theoremstyle{remark}

\numberwithin{equation}{section}

\begin{document}

\title{ The House of a Reciprocal Algebraic Integer}


%    author one information
\author{Dragan Stankov}
%\address{Katedra Matematike RGF-a 11000 Beograd, Dju\v{s}ina 7, Serbia}
\curraddr{}
\email{dstankov@rgf.bg.ac.rs}
\thanks{}

%    author two information
%\author{}
%\address{}
%\curraddr{}
%\email{}
%\thanks{}

%    \subjclass is required.

\date{}

\dedicatory{}

%    Abstract is required.
\begin{abstract}
Let $\alpha$ be an algebraic integer of degree $d$, which is reciprocal. The house of $\alpha$ is the largest modulus of its conjugates. We compute the minimum of the houses of all reciprocal algebraic integers of degree $d$ which are not roots of unity, say $\mathrm{mr}(d)$, for $d$ at most 34. We prove several lemmata and use them to avoid unnecessary calculations. The computations suggest several conjectures. The direct consequence of the last one is the conjecture of Schinzel and Zassenhaus.
We demonstrate the utility of $d$-th power of the house of $\alpha$.
\end{abstract}

\maketitle
% It is clear that α ≥ 1 and, from a
%classical theorem of L. Kronecker [K], it follows that α = 1 if and only if α is a root of unity.
%In 1965, A. Schinzel and H. Zassenhaus [SZ] conjectured that there exists a constant c > 0 such
%that if α is not a root of unity then α ≥ 1 + c/d. In 1985, a result of C.J. Smyth [S1] led
%D. Boyd [B] to conjecture that c should be equal to 3/2 log θ0 where θ0 = 1.324717 . . . is the
%smallest Pisot number, the real root of the polynomial x
%3 − x − 1. P. Voutier [V] proved that,
%if α is an algebraic integer of degree d ≥ 3, not a root of unity, then
%α ≥ 1 +1 2d (log log d/ log d)3

\section{Introduction}
Let $\alpha$ be an algebraic integer of degree $d$, with conjugates
$\alpha=\alpha_1, \alpha_2,\ldots,\alpha_d$ and minimal
polynomial $P$. The house of $\alpha$ (and of $P$) is defined by:
\[\house{\alpha} = \max\limits_{1\leq i\leq d}|\alpha_i|.\]
The Mahler measure of $\alpha$ is $M(\alpha) = \prod_{i=1}^{d}
\max(1, |\alpha_i|)$.
Clearly,
$\house{\alpha} > 1$, and a theorem of Kronecker \cite{K} tells us that $\house{\alpha} = 1$ if and only if $\alpha$ is a root
of unity. In 1965, Schinzel and Zassenhaus \cite{SZ} have made the following conjecture:
\begin{conjecture}[SZ]
There is a constant $c > 0$ such that if $\alpha$ is not a root of unity, then $\house{\alpha}\ge 1 + c/d$.
\end{conjecture}

Let $\mathrm{m}(d)$ denote the minimum of $\house{\alpha}$ over $\alpha$ of degree $d$ which are not roots of unity.
Let an $\alpha$ attaining $\mathrm{m}(d)$ be called extremal.
In 1985, D. Boyd \cite{B} conjectured, using a result of C.J. Smyth \cite{S1}, that $c$ should be equal to $3/2 \log \theta$ where $\theta = 1.324717 \ldots$ is the smallest Pisot number, the real root of the polynomial $x^3 - x - 1$. Intending to verify his conjecture that extremal $\alpha$ are always nonreciprocal, Boyd has computed the smallest houses for reciprocal polynomials of even degrees $\le 16$. We continued his computation with even degrees $\le 34$. So our Table \ref{table:nu} is the extension of Boyd's Table 2.

%\begin{conjecture}[Lind-Boyd]
%The smallest Perron number $\alpha$ of degree $n > 2$ has minimal polynomial

%$x^n-x-1$ if $n \not \equiv 3,5$ (mod 6),

%$(x^{n+2}-x^4-1)/(x^2-x+1)$ if $n \equiv 3$ (mod 6),

%$(x^{n+2}-x^2-1)/(x^2-x+1)$ if $n \equiv 5$ (mod 6).
%\end{conjecture}

%Wu \cite{Wu} gave all Perron numbers of degree $13 \le n \le 24$ with $\alpha \leq (2 +
%1/n)^{1/n}$ and their minimal polynomials, and verified that all the smallest Perron
%numbers of degree $13 \le n \le 24$ satisfy the conjecture of Lind-Boyd.

%In this paper we describe the computation of the minimum of $\overline{|\alpha |}$ for $\alpha$ of degree $d$, with $d \le 34$.
Let $\mathrm{mr}(d)$ denote the minimum of $\house{\alpha}$ over reciprocal $\alpha$ of degree $d$ which are not roots of
unity. %It is easy to see that this is an attained minimum.
Let an $\alpha$ attaining $\mathrm{mr}(d)$ be called extremal reciprocal. A polynomial $P(x)$ is primitive if it cannot be expressed as a polynomial in $x^k$, for some
$k \ge 2$. Clearly, any polynomial of degree $2p$ has to be primitive.
It is easy to verify that $\house{P(x^k)} = \sqrt[k]{\house{P(x)}}$.
Then our computations, as summarized in Table \ref{table:nu}, suggest the following:

\begin{conjecture}\label{sec:Sqrt}
%If $d=2k$ where $k$ is a composite and $d>12$ then $R(x)$ is not primitive.
If $d \ge 8$ is even and extremal reciprocal $\alpha$ of degree $d$ has minimal polynomial $R_d(x)$ then $\sqrt{\alpha}$ is extremal reciprocal of degree $2d$ and $R_{2d}(x)=R_d(x^2)$.
\end{conjecture}

%\begin{conjecture}[]
%%If $d\ge 14$ then
%$\mathrm{mr}(d)>\mathrm{mr}(d+4)$.
%\end{conjecture}

%\begin{conjecture}[]
%If $d \ge 8$ and there is a prime $p$ such that $d=2p$ then extremal reciprocal $\alpha$ has minimal polynomial with coefficients $-1,0,1$.
%\end{conjecture}

%\begin{conjecture}[]
%If $d\ge 14$ and there is a prime $p$ such that $d=2p$ then $\mathrm{mr}(d)>\mathrm{mr}(d-2)$ and $\mathrm{mr}(d)>\mathrm{mr}(d+2)$.
%\end{conjecture}

E. M. Matveev \cite{M} proved the following result:

\begin{theorem}\label{sec:Mat} Let $\alpha$ be an algebraic integer, not a root of unity, and let $d =
 deg(\alpha) \ge 2$. Then
\begin{equation}\label{Mat:1}
\house{\alpha}\ge \exp(\log(d + 0.5)/d^2).
\end{equation}
 Moreover, if $\alpha$ is reciprocal and $d\ge 6$, then
\begin{equation}\label{Mat:2}
\house{\alpha}\ge \exp(3\log(d/2)/d^2).
\end{equation}
\end{theorem}

Let $\sigma=1.169283\ldots$ be the house of $x^8+x^5+x^4+x^3+1$. It is proved in Table 1 that $\sigma$ is extremal reciprocal for $d=8$.
Then we obtain the following consequence of Conjecture \ref{sec:Sqrt}: if $d=2^k$ and $k\ge 3$ then 
\begin{equation}\label{Mat:3}
\mathrm{mr}(d)=\sigma^{8/d}.
\end{equation}
It is not hard to show that we get from \eqref{Mat:3} a better low bound than from \eqref{Mat:2} i.e.
$\sigma^{8/2^k}>\exp(3\log(d/2)/d^2)=(2^{k-1})^{3/(2^{2k})}$. Since $\sigma^5>2$ it follows that
$\sigma^{8/2^k}>2^{8/(5\cdot 2^k)}$. It remains to be shown that
$2^{8/(5\cdot 2^k)}>(2^{k-1})^{3/(2^{2k})}$ which is, however, equivalent with a true inequality $2^{k+3}>15(k-1)$.

The following lemmata can help us to avoid unnecessary calculations.

\begin{lemma}\label{sec:PolType}
If $d\ge 10$ and $P(x)$ is a reciprocal polynomial with coefficients $-1,0,1$ of degree $d$ such that
\begin{equation}\label{Pol:1}
P(x)=x^d-x^{d-1}-x^{d-2}-x^{d-3}-mx^{d-4}+\sum_{k=5}^{d-5}a_{d-k}x^{d-k}-mx^4-x^3-x^2-x+1,
%P(x)=1-\sum_{k=1}^{4}x^{k}+\sum_{k=5}^{d-5}a_{k}x^{k}+\sum_{k=d-4}^{d-1}x^{k}+x^d
\end{equation}
$m\in\{0,1\}$, then $P(x)$ has a real root $\alpha$ greater than $3/2$ and, consequently, $\house{\alpha}\ge 3/2$.
\end{lemma}

\begin{proof}[Proof]
It is obvious that $P(2)\ge 2^d+1-\sum_{k=1}^{d-1}2^k=2^d+1-2(2^{d-1}-1)=3>0$ so that the theorem will be proved if we show that $P(1.5)<0$.
\begin{eqnarray*}
P(1.5)&\le & 1.5^d-1.5^{d-1}-1.5^{d-2}-1.5^{d-3}+\\
& &+\;\sum_{k=5}^{d-5}1.5^{d-k}-1.5^3-1.5^2-1.5+1\\
&=&1.5^d-1.5^{d-1}-1.5^{d-2}-1.5^{d-3}+\\
& &+\;1.5^5\cdot 2(1.5^{d-9}-1)-1.5^3-1.5^2-1.5+1\\
&=&1.5^d-1.5^{d-1}-1.5^{d-2}-1.5^{d-3}+\\
& &+\;2\cdot 1.5^{d-4}-2\cdot 1.5^5-1.5^3-1.5^2-1.5+1\\
&=&1.5^{d-4}(1.5^4-1.5^{3}-1.5^{2}-1.5+2)-21.3125\\
&=&1.5^{d-4}(-0.0625)-21.3125\\
&<&0.%\housealpha
\end{eqnarray*}

\end{proof}

\begin{lemma}\label{sec:PolType2}
If $d\ge 6$ and $P(x)$ is a reciprocal polynomial with coefficients $-2$, $-1$, $0$, $1$, $2$ of degree $d$ such that
\begin{equation}\label{Pol:2}
P(x)=x^d-2x^{d-1}-2x^{d-2}+\sum_{k=3}^{d-3}a_{d-k}x^{d-k}-2x^2-2x+1,
%P(x)=1-\sum_{k=1}^{4}x^{k}+\sum_{k=5}^{d-5}a_{k}x^{k}+\sum_{k=d-4}^{d-1}x^{k}+x^d
\end{equation}
then $P(x)$ has a real root $\alpha$ greater than $2$ and, consequently, $\house{\alpha}\ge 2$.
\end{lemma}

\begin{proof}[Proof]
It is obvious that $P(3)\ge 3^d+1-2\sum_{k=1}^{d-1}3^k=3^d+1-3(3^{d-1}-1)=4>0$ so that the theorem will be proved if we show that $P(2)<0$.
\begin{eqnarray*}
P(2)&\le & 2^d-2\cdot 2^{d-1}-2\cdot 2^{d-2}+\\
& &+\;2\sum_{k=3}^{d-3}2^{d-k}-2\cdot 2^2-2\cdot 2+1\\
&=&-2^{d-1}+2^4(2^{d-5}-1)-2\cdot 2^2-2\cdot 2+1\\
&=&-27\\
&<&0.%\housealpha
\end{eqnarray*}

\end{proof}

\begin{lemma}\label{sec:PolType3}
If $d\ge 10$ and $P(x)$ is a reciprocal polynomial with coefficients $-2$, $-1$, $0$, $1$, $2$ of degree $d$ such that
\begin{equation}\label{Pol:3}
P(x)=x^d-2x^{d-1}-x^{d-2}-mx^{d-3}+\sum_{k=4}^{d-4}a_{d-k}x^{d-k}-mx^3-x^2-2x+1,
%P(x)=1-\sum_{k=1}^{4}x^{k}+\sum_{k=5}^{d-5}a_{k}x^{k}+\sum_{k=d-4}^{d-1}x^{k}+x^d
\end{equation}
$m\in\{1,2\}$, then $P(x)$ has a real root $\alpha$ greater than $2$ and, consequently, $\house{\alpha}\ge 2$.
\end{lemma}

\begin{proof}[Proof]
At first we show that $P(3)$ is positive:
\begin{eqnarray*}
P(3)&\ge& 3^d+3^{d-2}+3^2+1-2\sum_{k=1}^{d-1}3^k\\
&=&3^d+3^{d-2}+3^2+1-3(3^{d-1}-1)\\
&=&3^{d-2}+3^2+3+1\\
&>&0.
\end{eqnarray*}
The theorem will be proved if we show that $P(2)$ is negative:
\begin{eqnarray*}
P(2)&\le & 2^d-2\cdot 2^{d-1}-2^{d-2}-2^{d-3}+\\
& &+\;2\sum_{k=4}^{d-4}2^{d-k}-2^3-2^2-2\cdot 2+1\\
&=&-2^{d-2}-2^{d-3}+2^4\cdot 2(2^{d-7}-1)-2^3-2^2-4+1\\
&=&-2^{d-3}-2^5-15\\
&<&0.%\housealpha
\end{eqnarray*}

\end{proof}

The obvious consequence of the lemma \ref{sec:PolType} is: Mahler measure of a polynomial of type \eqref{Pol:1}, \eqref{Pol:2}, \eqref{Pol:3} is greater than $3/2$.
So if we have to find polynomials of small Mahler measure we can omit polynomials of these types.% \eqref{Pol:1}.

\section{Polynomials of composite and prime degrees}

Table \ref{table:theta} of Rhin and Wu suggests that if $d\ge 9$ is a composite number then $P_d(x)$ is a nonprimitive polynomial. We add the column $\mathrm{m}^{d}(d)$ to the table which is necessary to present the following lemma, corollary and conjecture.

\begin{lemma}\label{sec:PolProd}
Let $\mathrm{m}(d)$ is attained for $\alpha_d$ with minimal polynomial $P_d(x)$. If $\mathrm{m}^{d_1}(d_1)<\mathrm{m}^{d_2}(d_2)$ then the house of $P_{d_1}(x^{d_2})$ is less than the house of $P_{d_2}(x^{d_1})$.
\end{lemma}
\begin{proof}
If $\mathrm{m}^{d_1}(d_1)<\mathrm{m}^{d_2}(d_2)$ then $\mathrm{m}^{1/d_2}(d_1)<\mathrm{m}^{1/d_1}(d_2)$. Finally we should recall that the house of $P_{d_1}(x^{d_2})$ is equal to $\mathrm{m}^{1/d_2}(d_1)$ and the house of $P_{d_2}(x^{d_1})$ is equal to $\mathrm{m}^{1/d_1}(d_2)$.
\end{proof}

\begin{corollary}\label{sec:composite}
Let $d$ be a composite natural number. Let $\mathrm{m}(b_i)$ is attained for $\alpha_{b_i}$ with minimal polynomial $P_{b_i}(x)$ where $1 \le b_i <d$, are natural numbers which are divisors of $d$ such that $P_{b_i}(x)$ is a primitive polynomial, $i=1,2\ldots,k$.
If $\mathrm{m}^{b_1}(b_1)<\mathrm{m}^{b_2}(b_2)<\cdots<\mathrm{m}^{b_k}(b_k)$ then the nonprimitive polynomial $P_{b_1}(x^{d/b_1})$ has the house which is less than the house of any other nonprimitive polynomial of degree $d$.
\end{corollary}
\begin{proof}
The claim follows straightforwardly from Lemma \ref{sec:PolProd}.
\end{proof}

If $p$ is a prime number then it is obvious that the minimal polynomial of the extremal of degree $p$ is primitive or $P_1(x^p)=x^p-2$. Table \ref{table:theta} of Rhin and Wu suggests that $P_4(x)=x^4+x^3+1$ and $P_8(x)=x^8+x^7+x^4-x^2+1$ are the only primitive minimal polynomials of an extremal of a composite degree.

\begin{conjecture}\label{sec:compositeConj}
Let $d$ be a composite natural number and let $p_1,p_2,\ldots,p_k$ be odd prime numbers which are divisors of $d$ or $p_i=t$, $i=1,2,\ldots,k$ where $t$ is defined on the following manner:

%if $2 \mid d$ and $4 \nmid d$ then $t=2$;
$t:=1$;

if $4 \mid d$ and $8 \nmid d$ then $t:=4$;

if $8 \mid d$ then $t:=8$.

\noindent Let $\mathrm{m}^{p_1}(p_1)<\mathrm{m}^{p_2}(p_2)<\cdots<\mathrm{m}^{p_k}(p_k)\le 2$. If $P_{p_i}(x)$ is the minimal polynomial of the extremal of degree $p_i$ then $P_d(x)=P_{p_1}(x^{d/p_1})$ and $\mathrm{m}(d)=\house{P_{p_1}(x^{d/p_1})}$.
\end{conjecture}

If the previous conjecture is true then we just need to determine $\mathrm{m}(d)$ for $d$ is a prime. If $d$ is a composite number we can easily calculate $\mathrm{m}(d)=\mathrm{m}^{p_1/d}(p_1)$ where $p_1$ is determined as in the Conjecture \ref{sec:compositeConj}.

\begin{lemma}\label{sec:prime5Mod6}
Let $d\ge 5$ be a natural number such that $d\equiv 5\;(\mathrm{mod}\;6)$. Let $P_d(x)$ be defined
\begin{equation}\label{prime5Mod6:1}
P_d(x):=(x^{d+2}-x^2-1)/(x^2-x+1).
\end{equation}
Then $P_d(x)$ is a polynomial which has a real root $1<a_d<\sqrt[d]{2}$ such that
$$\lim_{d\rightarrow\infty}a_d^d= 2.$$
\end{lemma}
\begin{proof}
It can be proved by the mathematical induction that $$P_d(x)=(x^5+x^4-x^2-x)(x^{d-5}+x^{d-11}+\cdots+1)-1.$$
It is obvious that $P_d(1)=-1$ and $P_d(\sqrt[d]{2})=(2(\sqrt[d]{2})^2-(\sqrt[d]{2})^2-1)/((\sqrt[d]{2})^2-\sqrt[d]{2}+1)>0$. Hence, there is a real root $a_d\in(1,\sqrt[d]{2})$ of $P_d(x)$. It follows from $a_d^{d+2}-a_d^2-1=0$ that $a_d^{d}=1+1/a_d^2$. Since $a_d$ tends to $1$ we conclude that $a_d^{d}$ tends to $1+1/1^2=2$ when $d$ tends to $\infty$.
\end{proof}

It is proved in \cite{B} that $a_d$ is greater than any of its conjugates. Hence $a_d=\house{P_d(x)}$.
Table \ref{table:theta} of Rhin and Wu for $d=17$ and $d=23$ and the last two lemmata suggest the following
\begin{conjecture}\label{sec:primeConj5Mod6}
The extremal $\alpha$ of prime degree $d\ge 17$ such that $d\equiv 5\;(\mathrm{mod}\;6)$ has the minimal
polynomial $P_d(x)$ defined in Lemma \ref{sec:prime5Mod6}.
\end{conjecture}

Our attempt to generalize the minimal polynomial $(x^{22}-x^{11}-x+1)/((x-1)(x^2+1))$ of extremal $\alpha$ of degree $d= 19$ to prime degree $d\ge 19$ such that $d\equiv 7\;(\mathrm{mod}\;12)$ failed because the house of
$(x^{d+3}-x^{(d+3)/2}-x+1)/((x-1)(x^2+1))$ is greater than $\sqrt[d]{2}$ when $d\ge 19$. Therefore the next question arises: is there any prime $d\ne 2$ such that $\mathrm{m}^d(d)=2$? And, if there is such $d$ then whether $\mathrm{m}^{d^2}(d^2)<2$? These questions are closely related with the next

\begin{lemma}\label{sec:mddBounded}
The sequence $(\mathrm{m}^d(d))_{d\ge 1}$ is bounded and $2$ is an upper bound.
\end{lemma}
\begin{proof}
If $\mathrm{m}(d)$ is attained for $\alpha_d$ then $\mathrm{m}(d)=\house{\alpha_d}\le \house{x^d-2}= \sqrt[d]{2}$. The claim follows straightforwardly if we raise both sides of the inequality to the power $d$.
\end{proof}

\begin{corollary}\label{sec:mddAccumulation}
The sequence $(\mathrm{m}^d(d))_{d\ge 1}$ has an accumulation point in $[1,2]$.
\end{corollary}
\begin{proof}
The claim is direct consequence of lemma \ref{sec:mddBounded} and the Bolzano-Weierstrass Theorem.
\end{proof}

The last few lemmata and corollaries show that $\house{\alpha}^d$ can play an important role
in the research of algebraic integers analogously to the Mahler measure. The obvious benefit is that we can exclude nonprimitive polynomials because $\house{\alpha}^d=\house{\sqrt[k]{\alpha}}^{kd}$. Also it is interesting to ask the Lehmer question whether there exists a positive number $\epsilon$  such that
if $\alpha$ is neither 0 nor a root of unity, then $\house{\alpha}^d \ge 1+\epsilon$.
Is there an accumulation point less than $2$ of the sequence $(\mathrm{m}^d(d))_{d\ge 1}$ is another interesting question. We suggest $\house{\alpha}^d$ to be called the \textbf{powerhouse} of $\alpha$ and denoted with $\mathrm{ph}(\alpha)$.

%\roof{\alpha_d}

\begin{lemma}\label{sec:PolProdRec}
Let $\mathrm{mr}(d)$ is attained for $\alpha_d$ with minimal reciprocal polynomial $R_d(x)$. Let $k_1$, $k_2$ be integers and $d_1$, $d_2$ be even integers such that $k_1d_1=k_2d_2$. If $\mathrm{mr}^{d_1}(d_1)<\mathrm{mr}^{d_2}(d_2)$ then the house of $R_{d_1}(x^{k_1})$ is less than the house of $R_{d_2}(x^{k_2})$.
\end{lemma}
\begin{proof}
If $\mathrm{mr}^{d_1}(d_1)<\mathrm{mr}^{d_2}(d_2)$ then $\mathrm{mr}^{1/k_1}(d_1)<\mathrm{mr}^{1/k_2}(d_2)$. It remains to recall that the house of $R_{d_1}(x^{k_1})$ is equal to $\mathrm{mr}^{1/k_1}(d_1)$ and the house of $R_{d_2}(x^{k_2})$ is equal to $\mathrm{mr}^{1/k_2}(d_2)$.
\end{proof}

\begin{corollary}\label{sec:compositeRec}
Let $d/2$ be a composite natural number. Let $\mathrm{mr}(b_i)$ is attained for a reciprocal $\alpha_{b_i}$ with minimal polynomial $R_{b_i}(x)$ where $1 \le b_i <d$, are natural numbers which are divisors of $d$ such that $R_{b_i}(x)$ is a primitive polynomial, $i=1,2,\ldots,k$.
If $\mathrm{mr}^{b_1}(b_1)<\mathrm{mr}^{b_2}(b_2)<\cdots<\mathrm{mr}^{b_k}(b_k)$ then the nonprimitive polynomial $P_{b_1}(x^{d/b_1})$ has the house which is less than the house of any other nonprimitive polynomial of degree $d$.
\end{corollary}
\begin{proof}
The claim follows straightforwardly from Lemma \ref{sec:PolProdRec}.
\end{proof}

If $p$ is a prime number then it is obvious that the minimal polynomial of the extremal reciprocal of degree $2p$ is primitive or $R_2(x^p)=x^{2p}+3x^p+1$. Table \ref{table:nu} suggests that $R_8(x)%=x^8+x^5+x^4+x^3+1
$, $R_{12}(x)%=x^{12}+x^{11}+x^{10}-x^8-x^7-x^6-x^5-x^4+x^2+x+1
$, and $R_{18}%(x)=x^{18}+x^{16}+x^{15}+x^{14}+2x^{13}+x^{12}+2x^{11}+2x^{10}+x^9+2x^8+2x^7+x^6+2x^5+x^4+x^3+x^2+1
$ are the only primitive minimal polynomials of an extremal reciprocal of a degree $d$ such that $d/2$ is a composite number.

\begin{conjecture}\label{sec:compositeConjRec}
Let $d$ be a composite natural number and let $p_1,p_2,\ldots,p_k$ be odd prime numbers which are divisors of $d$ or $p_i\in\{s,t\}$, $i=1,2,\ldots,k$ where $s=9$ and $t$ is defined on the following manner:

%if $2 \mid d$ and $4 \nmid d$ then $t=2$;
$t:=1$;

if $8 \mid d$ and $12 \nmid d$ then $t=4$;

if $12 \mid d$ then $t=6$.

%\noindent Let $\mathrm{m}^{p_1}(p_1)<\mathrm{m}^{p_2}(p_2)<\cdots<\mathrm{m}^{p_k}(p_k)\le 2$. If $P_{p_i}(x)$ is the minimal polynomial of the extremal of degree $p_i$ then $\mathrm{m}(d)=\house{P_{p_1}(x^{d/p_1})}$.
Let $\mathrm{mr}^{2p_1}(2p_1)<\mathrm{mr}^{2p_2}(2p_2)<\cdots<\mathrm{mr}^{2p_k}(2p_k)$. 
If $R_{2p_i}(x)$ is the minimal polynomial of the extremal of degree $2p_i$ then $R_{d}(x)=R_{2p_1}(x^{d/(2p_1)})$ and $\mathrm{mr}(d)=\house{R_{2p_1}(x^{d/(2p_1)})}$.
%If $P_{2p_i}(x)$ is the minimal polynomial of the extremal reciprocal of degree $2p_i$ then the nonprimitive polynomial $P_{2p_1}(x^{d/(2p_1)})$ has the house which is less than the house of any other nonprimitive polynomial of degree $d$.
\end{conjecture}

If the previous conjecture is true then we just need to determine $\mathrm{mr}(d)$ for $d/2$ is a prime number. If $d/2$ is a composite number we can easily calculate $\mathrm{mr}(d)=\mathrm{mr}^{p_1/d}(p_1)$ where $p_1$ is determined as in the Conjecture \ref{sec:compositeConjRec}. %The next question is interesting: is there any prime $d>2$ such that $\mathrm{m}^d(d)=2$? And if there is such $d$ then is $\mathrm{m}^{d^2}(d^2)<2$? %The positive answer of these questions demand the

\begin{lemma}\label{sec:mddBoundedRec}
The sequence $(\mathrm{mr}^d(d))_{d\ge 1}$ is bounded and $U=6.854102\ldots$ is an upper bound.
\end{lemma}
\begin{proof}
If $\mathrm{mr}(d)$ is attained for $\alpha_d$ then $$\mathrm{mr}(d)=\house{\alpha_d}\le \house{x^d+3x^{d/2}+1}= \sqrt[d/2]{2.618\ldots}.$$ The claim follows straightforwardly if we raise both sides of the inequality to the power $d$.
\end{proof}

\begin{corollary}\label{sec:mddAccumulationRec}
In the interval $[1,U]$ there is an accumulation point of the sequence $(\mathrm{mr}^d(d))_{d\ge 1}$.
\end{corollary}
\begin{proof}
The claim is direct consequence of lemma \ref{sec:mddBoundedRec} and the Bolzano-Weierstrass Theorem.
\end{proof}

\section{Results}

%\section{Tabels}

%$\overline{|\alpha |}$
\begin{table}[!htbp]
\caption{Extreme values of $\house{\alpha}$ for reciprocal $\alpha$ of even degree $d\le 34$. The minimum $\mathrm{mr}(d)$ is attained for an $\alpha$ with minimal polynomial $R_d(x)$ having $\nu$ roots outside the unit circle.} % title of Table $\house(\alpha)$
\label{table:nu} % is used to refer this table in the text
%\centering % used for centering table
\begin{tabular}{rllllll} % centered columns (4 columns)
\hline\noalign{\smallskip}
$d$ &	$\nu$ &	$\mathrm{mr}(d)$ & $R_d(x)$ \\ [0.5ex] % inserts table
%heading
\noalign{\smallskip}\hline\noalign{\smallskip}
 2 & 1 & 2.61803398874989 &  1  3  \\
 4 & 2 & 1.53922233842043 &  1  1  3 \\
 6 & 2 & 1.32166315615906 &  1  2  2  1\\
 8 & 2 & 1.16928302978955 &  1  0  0  1  1\\
10 & 2 & 1.12571482154239 &  1  0  1  1  0  1\\
12 & 2 & 1.10805485364877 &  1  1  1  0 -1 -1 -1\\
14 & 4 & 1.09390168574961 &  1  0  0  0  1  1  0  1\\
16 & 4 & 1.08133391225354 &  $R_8(x^2)$\\ %1  0  0  0  0  0  1  0  1\\
18 & 4 & 1.07185072135591 &  1  0  1  1  1  2  1  2  2  1\\
20 & 4 & 1.06099708837602 &  $R_{10}(x^2)$\\ %1  0  0  0  1  0  1  0  0  0  1\\
22 & 4 & 1.06621758541355 &  1  1  0 -1  0  0  0  0  0 -1  0  1\\
24 & 4 & 1.05264184490679 &  $R_{12}(x^2)$\\ %1  0  1  0  1  0  0  0 -1  0 -1  0 -1\\
26 & 8 & 1.05784846909829 &  1  0  0  1  0 -1  0  0 -1 -1  1  0  0  2\\ %26 & 8 & 1.05968760806902 &  1  1  0  0  0  0  1  0 -1  0  1  1  1  1\\
28 & 8 & 1.04589755031246 &  $R_{14}(x^2)$\\ %1  0  0  0  0  0  0  0  1  0  1  0  0  0  1\\
30 & 6 & 1.04026214469874 &  $R_{10}(x^3)$\\
32 & 8 & 1.03987206532993 &  $R_8(x^4)$\\ %1  0  0  0  0  0  0  0  0  0  0  0  1  0  0  0  1\\
34 & 8 & 1.04961810533324 &  1  0  1  1  0  1  0  0  0  0  0  0  1  0  1  1  0  1\\
\noalign{\smallskip}\hline
\end{tabular}
\end{table}

\begin{table}[!htbp]
\caption{Let $\theta = 1. 3247\ldots$ is the real root
 of $x^3 - x-1$. Extreme values of $\house{\alpha}$ for $\alpha$ of degree $d\le 28$, calculated by Rhin and Wu \cite{RW2}, are greater than or equal to $\theta^{3/(2d)}$.} % title of Table $\house(\alpha)$
\label{table:theta} % is used to refer this table in the text
%\centering % used for centering table
\begin{tabular}{rllllll} % centered columns (4 columns)
\hline\noalign{\smallskip}
$d$ &	$\mathrm{m}(d)$ & & $\theta^{3/(2d)}$ & $\mathrm{m}^d(d)$ & coefficients of $P_d(x)$\\ [0.5ex] % inserts table
%heading
\noalign{\smallskip}\hline\noalign{\smallskip}
  1 & 2 & $>$ & 1.524703 & 2 & 1 -2\\
  2 & 1.41421356 & $>$ & 1.234788 & 2 & 1 0 -2\\
  3 & 1.15096392 & $=$ & 1.150964 & 1.524703 & 1 1 0 -1\\
  4 & 1.18375181 & $>$ & 1.111210 & 1.963553 & 1 1 0 0 1\\
  5 & 1.12164517 & $>$ & 1.088020 & 1.775323 & 1 0 -1 -1 1 1\\
  6 & 1.07282986 & $=$ & 1.072830 & 1.524703 & $P_3(x^2)$\\
  7 & 1.09284559 & $>$ & 1.062110 & 1.861708 & 1 1 0 0 1 -1 -1\\
  8 & 1.07562047 & $>$ & 1.054140 & 1.791730 & 1 1 0 0 1 0 -1 0 1\\
  9 & 1.04798219 & $=$ & 1.047982 & 1.524703 & $P_3(x^3)$\\
 10 & 1.05907751 & $>$ & 1.043082 & 1.775323 & $P_5(x^2)$\\
 11 & 1.05712485 & $>$ & 1.039090 & 1.842422 & 1 1 0 0 1 1 0 -1 0 1 0 -1\\
 12 & 1.03577500 & $=$ & 1.035775 & 1.524703 & $P_3(x^4)$\\
 13 & 1.05372001 & $>$ & 1.032978 & 1.974367 & 1 0 -1 0 1 0 -1 -1 1 1 -1 -1 1 1\\
 14 & 1.04539255 & $>$ & 1.030587 & 1.861708 & $P_7(x^2)$\\
 15 & 1.02851905 & $=$ & 1.028519 & 1.524703 & $P_3(x^5)$\\
 16 & 1.03712124 & $>$ & 1.026713 & 1.791730 & $P_8(x^2)$\\
 17 & 1.03930211 & $>$ & 1.025122 & 1.925798 & 1 1 0 -1 -1 0 1 1 0 -1 -1 0\\
 &&&&& 1 1 0 -1-1-1\\
 18 & 1.02371001 & $=$ & 1.023710 & 1.524702 & $P_3(x^6)$\\
 19 & 1.03641032 & $>$ & 1.022448 & 1.972890 & 1 1 0 0 1 1 0 0 1 1 0 -1 0\\
 &&&&& 1 0 -1 0 1 0 -1\\
 20 & 1.02911491 & $>$ & 1.021314 & 1.775323 & $P_5(x^4)$\\
 21 & 1.02028875 & $=$ & 1.020289 & 1.524703 & $P_3(x^7)$\\
 22 & 1.02816577 & $>$ & 1.019358 & 1.842422 & $P_{11}(x^2)$\\
 23 & 1.02932014 & $>$ & 1.018508 & 1.943841 & 1 1 0 -1 -1 0 1 1 0 -1 -1 0\\
 &&&&& 1 1 0 -1 -1 0 1 1 0 -1-1-1\\
 24 & 1.01773032 & $=$ & 1.017730 & 1.524703 & $P_3(x^8)$\\
 25 & 1.02322489 & $>$ & 1.017015 & 1.775323 & $P_5(x^5)$\\
 26 & 1.02650865 & $>$ & 1.016355 & 1.974367 & $P_{13}(x^2)$\\
 27 & 1.01574486 & $=$ & 1.015745 & 1.524703 & $P_3(x^9)$\\
 28 & 1.02244440 & $>$ & 1.015178 & 1.861708 & $P_7(x^4)$\\
\noalign{\smallskip}\hline
\end{tabular}
\end{table}

In the Table \ref{table:nu} we listed irreducible, reciprocal, integer polynomials with even degree at most 34 having the smallest house.
We add the column $\theta^{3/(2d)}$ to the Table \ref{table:theta} so that it suggests the following
\begin{conjecture}[SZB]
There is a constant $T > 1$ such that if $\alpha$ is not a root of unity, then $\house{\alpha}\ge T^{1/d}$.
\end{conjecture}
It is easy to show that the conjecture of Schinzel and Zassenhaus is a direct consequence of
the previous conjecture.  In order to expand $T^{1/d}$ as a Taylor series in $1/d$, we use the known Taylor series of function $T^x$. Thus \[T^{\frac{1 }{d}}=1+\frac{\log(T)}{d}+\frac{\log^2(T)}{2!d^2}+\cdots+\frac{\log^k(T)}{k!d^k}+\cdots.\]
If $T=\theta^{3/2}$ and if we take only two terms of the series we get the conjecture of Schinzel and Zassenhaus with Boyd's \cite{B} suggestion for $c$.

\begin{table}[!htbp]
\caption{Let $\tau=1.125715\ldots$ be the house of $x^{10}+x^8+x^7+x^5+x^3+x^2+1$. Extreme values of $\house{\alpha}$ for reciprocal $\alpha$ of even degree $d\le 34$ are greater than or equal to $\tau^{10/d}$.} % title of Table $\house(\alpha)$
\label{table:sigma} % is used to refer this table in the text
%\centering % used for centering table
\begin{tabular}{rllllll} % centered columns (4 columns)
\hline\noalign{\smallskip}
$d$ &	$\mathrm{mr}(d)$ & & $\tau^{10/d}$ & $\mathrm{mr}^d(d)$\\ [0.5ex] % inserts table
%heading
\noalign{\smallskip}\hline\noalign{\smallskip}
 2 & 2.618034 & $>$ & 1.807765 & 6.854102\\
 4 & 1.539222 & $>$ & 1.344531 & 5.613134\\
 6 & 1.321663 & $>$ & 1.218187 & 5.329969\\
 8 & 1.169283 & $>$ & 1.159539 & 3.494276\\
10 & 1.125715 & $=$ & 1.125715 & 3.268014\\
12 & 1.108055 & $>$ & 1.103715 & 3.425588\\
14 & 1.093902 & $>$ & 1.088265 & 3.513145\\
16 & 1.081334 & $>$ & 1.076819 & 3.494276\\ %1  0  0  0  0  0  1  0  1\\
18 & 1.071851 & $>$ & 1.068000 & 3.486723\\
20 & 1.060997 & $=$ & 1.060997 & 3.268014\\ %1  0  0  0  1  0  1  0  0  0  1\\
22 & 1.066218 & $>$ & 1.055301 & 4.098345\\
24 & 1.052642 & $>$ & 1.050578 & 3.425588\\ %1  0  1  0  1  0  0  0 -1  0 -1  0 -1\\
26 & 1.057848 & $>$ & 1.046599 & 4.514652\\
28 & 1.045898 & $>$ & 1.043199 & 3.513145\\ %1  0  0  0  0  0  0  0  1  0  1  0  0  0  1\\
30 & 1.040262 & $=$ & 1.040262 & 3.268014\\
32 & 1.039872 & $>$ & 1.037699 & 3.494276\\ %1  0  0  0  0  0  0  0  0  0  0  0  1  0  0  0  1\\
34 & 1.049618 & $>$ & 1.035443 & 5.188773\\
\noalign{\smallskip}\hline
\end{tabular}
\end{table}

%A polynomial $P(x)$ is primitive if it cannot be expressed as a polynomial in $x^k$, for some
%$k \ge 2$. It is easy to verify that $\house{P(x^k)} = \sqrt[k]{\house{P(x)}}$.
In the following tables we listed irreducible, reciprocal, integer polynomials with even degree at most 34 having small house. If $d=2p$ where $p$ is a prime number then all found polynomials are primitive, otherwise we marked a primitive polynomial with the symbol $P$.
A polynomial which has small Mahler measure, less than $1.3$, we marked with the symbol $M$. This list is only known to be complete through degree 20. If $d>22$ only polynomials of height one are completely investigated.

\begin{table}[!htbp]
\caption{Irreducible, reciprocal, integer polynomials with even degree at most 34 having small house.} % title of Table $\house(\alpha)$
\label{table:small} % is used to refer this table in the text
%\centering % used for centering table
\begin{tabular}{rlrllll} % centered columns (4 columns)
\hline\noalign{\smallskip}
$d$ &	House & Out & Coefficients \\ [0.5ex] % inserts table
%heading
\noalign{\smallskip}\hline\noalign{\smallskip}
 2 & 2.61803398874989 &  1 & 1  3  \\
\\
 4 & 1.53922233842043 &  2 &  1  1  3 P\\
 4 & 1.61803398874989 &  2 &  1  0  3 \\
\\
 6 &  1.32166315615906 &  2 &  1  2  2  1\\
 6 &  1.32471795724474 &  3 &  1  1 -1 -3\\
 6 &  1.33076841869444 &  2 &  1  1  0 -2\\
 6 &  1.33950727686218 &  2 &  1  0  2  1\\
\\
 8 &  1.16928302978955 &  2 &  1  0  0  1  1 P\\
 8 &  1.17050710134464 &  2 &  1  1  1  0 -1 P\\
 8 &  1.18375181855821 &  4 &  1  1  0  1  3 P\\
 8 &  1.21502361972591 &  2 &  1  2  2  1  1 P\\
 8 &  1.21962614693622 &  2 &  1  1  1  0  1 P\\
\\
10 &  1.12571482154239 &  2 &  1  0  1  1  0  1 M\\
10 &  1.13295293839656 &  2 &  1  1  0  0  0 -1 M\\
10 &  1.14208745799486 &  2 &  1  1  0 -1  0  0\\
10 &  1.16703006058662 &  2 &  1  0  0  0  1  1\\
10 &  1.17004216879649 &  2 &  1  0  0  1  0 -1\\
10 &  1.17628081825992 &  1 &  1  1  0 -1 -1 -1 M\\
\\
12 &  1.10805485364877 &  2 &  1  1  1  0 -1 -1 -1 PM\\
12 &  1.11850195225747 &  2 &  1  1  0  0  0 -1 -1 PM\\
12 &  1.12445269119837 &  2 &  1  0  1  0  0  1 -1 PM\\
12 &  1.12742072023266 &  4 &  1  1  1  1  1  2  3 P\\
12 &  1.12819252128504 &  2 &  1  0  1  1  1  2  1 PM\\
12 &  1.13861753595063 &  4 &  1  1  1  1  2  2  3 P\\
12 &  1.14103240247447 &  2 &  1  1  0 -1  0  1  1 P\\
12 &  1.14211801611167 &  2 &  1  0  0  1 -1 -1  1 P\\
12 &  1.14460531348308 &  2 &  1  2  2  1  1  1  1 P\\
\\
14 &  1.09390168574961 &  4 &  1  0  0  0  1  1  0  1\\
14 &  1.09663696733953 &  4 &  1  1  0  0  1  0 -1 -1\\
14 &  1.09873127474994 &  4 &  1  1  0 -1  0  1  1  1\\
14 &  1.10540085265079 &  3 &  1  1  1  0 -1 -1 -1 -1 M\\
14 &  1.10912255228309 &  4 &  1  0  0  0 -1  0  1  1\\
14 &  1.11020596746828 &  4 &  1  1  1  1  1  0  1  1\\
14 &  1.11132960322928 &  4 &  1  1  2  2  2  2  1  1\\
14 &  1.11141077514688 &  4 &  1  0 -1  0  1  1  0 -2\\
14 &  1.11157496383649 &  3 &  1  1  1  1  0 -1 -2 -3\\
\noalign{\smallskip}\hline
\end{tabular}
\end{table}

\begin{table}[!htbp]
\caption{Irreducible, reciprocal, integer polynomials with even degree at most 34 having small house. %This list is only known to be complete through degree 20. If $d>22$ only polynomials of height one are completely investigated.
} % title of Table $\house(\alpha)$
\label{table:small1} % is used to refer this table in the text
%\centering % used for centering table
\begin{tabular}{rlrllll} % centered columns (4 columns)
\hline\noalign{\smallskip}
$d$ &	House & Out & Coefficients \\ [0.5ex] % inserts table
%heading
\noalign{\smallskip}\hline\noalign{\smallskip}
16 &  1.08133391225354 &  4 &  1  0  0  0  0  0  1  0  1\\
16 &  1.08189976492494 &  4 &  1  0  1  0  1  0  0  0 -1\\
16 &  1.08568941631979 &  4 &  1  1  1  1  0  0 -1 -2 -1 P\\
16 &  1.08800359308148 &  8 &  1  0  1  0  0  0  1  0  3\\
16 &  1.09054731172112 &  4 &  1  0 -1 -1  0  2  1 -1 -1 P\\
16 &  1.09145310961609 &  4 &  1  0 -1  0  0  1  1 -1 -1 P\\
16 &  1.09341867119317 &  4 &  1  1  1  2  2  2  2  2  3 P\\
16 &  1.09441893214119 &  4 &  1  1  0  0  0  0  1  0 -1 P\\
\\
18 &  1.07185072135591 &  4 &  1  0  1  1  1  2  1  2  2  1 P\\
18 &  1.07715254391892 &  4 &  1  1  0 -1  0  0 -1 -1  1  2 P\\
18 &  1.08350235040111 &  4 &  1  0  1  0  0  1  0  1  1  1 P\\
18 &  1.08507352195696 &  4 &  1  0  0  1  1  0  0  2  1 -1 P\\
18 &  1.08914119715632 &  4 &  1  0  0  1  0  1  0  0  1 -1 P\\
18 &  1.09054731172112 &  4 &  1 -1  0  0  0  1 -1  0  1 -1 P\\
18 &  1.09059677435683 &  6 &  1  0  1  1  1  1  2  2  1  3 P\\
18 &  1.09151857842220 &  4 &  1  0  0  1  0  0  1 -1 -1  1 P\\
18 &  1.09217083605099 &  6 &  1  1  0 -1 -1  1  2  1  0 -1 P\\
18 &  1.09282468746958 &  3 &  1  1  0  0  0 -1 -1  0  0 -1 PM\\
18 &  1.09381566687105 &  4 &  1  0  1  1  1  1  1  1  1  0 P\\
\\
20 &  1.06099708837602 &  4 &  1  0  0  0  1  0  1  0  0  0  1\\
20 &  1.06440262043860 &  4 &  1  0  1  0  0  0  0  0  0  0 -1\\
20 &  1.06554639211891 &  4 &  1  1  0 -1 -1 -1 -1 -1  0  2  3 PM\\
20 &  1.06868491988746 &  4 &  1  0  1  0  0  0 -1  0  0  0  0\\
20 &  1.07086533169145 &  6 &  1  1  0 -1 -1 -1  0  0  0  1  2 P\\
20 &  1.07888517088957 &  8 &  1  1  0  0  1  1  0  1  2  0 -1 P\\
20 &  1.08029165533508 &  4 &  1  0  0  0  0  0  0  0  1  0  1\\
20 &  1.08081406854476 &  4 &  1  1  1  0  0  0  0 -1 -1  0  0 P\\
20 &  1.08093254714134 &  4 &  1  1  0  0  0 -1 -1  0  0  0  1 PM\\
20 &  1.08100667136043 &  4 &  1  1  0  0  0  0  0 -1  0  1  1 P\\
20 &  1.08111514762666 &  4 &  1  0  0  1  0  0  0  0 -1 -1  1 P\\
20 &  1.08168487499664 &  4 &  1  0  0  0  0  0  1  0  0  0 -1\\
20 &  1.08205695902988 &  4 &  1  1 -1 -1  1  0 -1  0  0  0  1 P\\
20 &  1.08213153864244 &  4 &  1  0 -1  1  0 -2  1  1 -1  0  1 P\\
20 &  1.08215867145095 &  6 &  1  1  1  0 -1 -2 -1 -1  1  1  1 P\\
20 &  1.08222782056950 &  4 &  1  1  0 -1  0  1  1  0  0  0  1 P\\
20 &  1.08228216492799 &  6 &  1  0 -1  0  1  1 -1  0  2  0 -1 P\\
20 &  1.08286885593631 &  6 &  1 -1  1  0  0  0  1 -1  0  1 -1 P\\
\noalign{\smallskip}\hline
\end{tabular}
\end{table}

\begin{table}[!htbp]
\caption{Irreducible, reciprocal, integer polynomials with even degree at most 34 having small house. %This list is only known to be complete through degree 20. If $d>22$ only polynomials of height one are completely investigated.
} % title of Table $\house(\alpha)$
\label{table:small2} % is used to refer this table in the text
%\centering % used for centering table
\begin{tabular}{rlrllll} % centered columns (4 columns)
\hline\noalign{\smallskip}
$d$ &	House & Out & Coefficients \\ [0.5ex] % inserts table
%heading
\noalign{\smallskip}\hline\noalign{\smallskip}
22 &  1.06621758541355 &  4 &  1  1  0 -1  0  0  0  0  0 -1  0  1 M\\
22 &  1.06827041313888 &  6 &  1  1  1  0 -1 -1  0  1  2  1 -1 -1\\
22 &  1.06843153438173 &  7 &  1  1  1  1  1  0  0 -1 -1 -2 -1 -1\\
22 &  1.06849271893547 &  7 &  1  0  1  0  1  0  0  0 -1  0 -2  1\\
22 &  1.06857505098600 &  8 &  1  0 -1  0  1  1 -1 -1  1  2  0 -3\\
22 &  1.07151243860039 &  4 &  1  0  1  1  0  1  0  1  1  1  2  1\\
22 &  1.07266460893982 &  5 &  1  1  1  0 -1 -1 -1  0  0  0  0 -1\\
22 &  1.07448519196034 &  6 &  1  1  0 -1  0  1  1  0  0  0  0 -1\\
22 &  1.07483796674177 &  4 &  1  1  0  0  0 -1 -1  0  0  0  1  1 M\\
22 &  1.07534302358553 &  4 &  1  1  1  0 -1 -1  0  0  1  0 -1 -1\\
22 &  1.07656105250927 &  6 &  1  0  0  1  0  0  0  1  1  0  1  1\\
22 &  1.07711967842550 &  9 &  1  0  1  0  1  0  1  1  0  2 -1  1\\
22 &  1.07719162675888 &  4 &  1  1  1  0  0 -1 -1 -1  0  1  2  2\\
22 &  1.07798582029041 &  7 &  1  1  1  1  0  0  0 -1  0 -1 -2 -2\\
%22 &  1.07835713596497 &  6 &  1  1  0  0  1  0 -1  0  1  1  1  1\\
%"22   1.07938210181976   6   1  1  1  2  2  2  2  2  3  2  2  3"
%"22   1.07952161387739   7   1  1  0  0  0  0  1  0  -2  -1  0  -1"
%"22   1.07968940628056   5   1  1  1  1  0  0  0  -1  0  -1  -2  -1"
%22 &  1.07972980409243 &  6 &  1  0  0  1  0  0  0  1  0  0  1 -1\\
%22 &  1.07995956978776 &  4 &  1  1  0 -1 -1  0  1  1  1  0 -1 -2\\
%22 &  1.08002360358639 &  4 &  1  0  1  0  1  0  1  0  0  1  0  1\\
%22 &  1.08031460018401 &  5 &  1  1  0  0  0 -1 -1  0  0 -1  0  1\\
%22 &  1.08044451039555 &  6 &  1  0  0  1  0 -1  1  0 -1  1  1  0\\
\\
24 &  1.05264184490679 &  4 &  1  0  1  0  1  0  0  0 -1  0 -1  0 -1\\
24 &  1.05351295923098 &  6 &  1  0  0  0  0  0  0  0  0  1  0  0  1\\
24 &  1.05388045665602 &  6 &  1  0  0  1  0  0  1  0  0  0  0  0 -1\\
24 &  1.05759252657036 &  4 &  1  0  1  0  0  0  0  0  0  0 -1  0 -1\\
24 &  1.05784057193422 & 12 &  1  0  0  1  0  0  0  0  0  1  0  0  3\\
24 &  1.06003424557321 &  4 &  1  0  1  1  0  2  0  1  1  0  1  0  1 PM\\
24 &  1.06040213654932 &  4 &  1  0  0  0  1  0  0  0  0  0  1  0 -1\\
24 &  1.06177436224626 &  6 &  1  1  1  1  0 -1 -2 -2 -2 -1  1  2  3 P\\
24 &  1.06180069703907 &  8 &  1  0  1  0  1  0  1  0  1  0  2  0  3\\
24 &  1.06216407455959 &  4 &  1  0  0  0  1  0  1  0  1  0  2  0  1\\
24 &  1.06490580489257 &  4 &  1  1  1  0  0  0  0 -1 -1  0  0  0 -1 P\\
24 &  1.06535138762690 &  4 &  1  1  1  0 -1 -1 -1  0  0  0  0  0  1 PM\\
24 &  1.06537706223156 &  6 &  1  1  0 -1 -2 -1  1  2  2  0 -1 -1 -1 P\\
%24 &  1.06620185021388 &  6 &  1  0  1  1  0  1  0  0  0  0  0  0  1\\
%24 &  1.06706023070426 &  8 &  1  0  1  0  1  0  1  0  2  0  2  0  3\\
%24 &  1.06707490593966 &  6 &  1  0  0  2  0  0  2  0  0  1  0  0  1\\
%24 &  1.06785515431593 &  4 &  1  1  1  0  0 -1  0  0  1  0  0 -1  0\\
\\
26 &  1.05784846909829 &  8 &  1  0  0  1  0 -1  0  0 -1 -1  1  0  0  2\\
26 &  1.05968760806902 &  8 &  1  1  0  0  0  0  1  0 -1  0  1  1  1  1\\
26 &  1.06184735727122 &  6 &  1  0  1  0  0  1  0  1  1  0  1  0  0  1\\
26 &  1.06277446310360 & 10 &  1  0  0  0  0  0  1  1  0  0  0 -1  1  1\\
26 &  1.06342599606179 &  6 &  1  1  0  0  0  0  0  0  1  1  1  0 -1 -1\\
26 &  1.06345648424260 &  6 &  1  0  0  0  0  1  1  0  0  0  1  1  0  1\\
26 &  1.06559111842191 &  6 &  1  1  0  0  1  0 -1  0  1 -1 -1  0  0 -1\\
26 &  1.06596578704523 &  6 &  1  1  1  0  0  0  1  1  1  0 -1 -1  0  0\\
26 &  1.06619413895030 &  6 &  1  1  0  0  0  0  0 -1  0  1  0  0  0 -1\\
26 &  1.06644267309866 &  6 &  1  0  0  0  0  1  0  1  1  0  0  0  1  1\\
26 &  1.06666365977337 &  4 &  1  1  0  0  0 -1 -1  0  1  1  1  1  0 -1 M\\

\noalign{\smallskip}\hline
\end{tabular}
\end{table}

\begin{table}[!htbp]
\caption{Irreducible, reciprocal, integer polynomials with even degree at most 34 having small house. %This list is only known to be complete through degree 20. If $d>22$ only polynomials of height one are completely investigated.
} % title of Table $\house(\alpha)$
\label{table:small3} % is used to refer this table in the text
%\centering % used for centering table
\begin{tabular}{rlrllll} % centered columns (4 columns)
\hline\noalign{\smallskip}
$d$ &	House & Out & Coefficients \\ [0.5ex] % inserts table
%heading
\noalign{\smallskip}\hline\noalign{\smallskip}
%\\
28 &  1.04589755031246 &  8 &  1  0  0  0  0  0  0  0  1  0  1  0  0  0  1\\
28 &  1.04720435796435 &  8 &  1  0  1  0  0  0  0  0  1  0  0  0 -1  0 -1\\
28 &  1.04820383263464 &  8 &  1  0  1  0  0  0 -1  0  0  0  1  0  1  0  1\\
28 &  1.05138045095521 &  6 &  1  0  1  0  1  0  0  0 -1  0 -1  0 -1  0 -1\\
28 &  1.05314887470058 &  8 &  1  0  0  0  0  0  0  0 -1  0  0  0  1  0  1\\
28 &  1.05366311858595 &  8 &  1  0  1  0  1  0  1  0  1  0  0  0  1  0  1\\
28 &  1.05419618820658 &  8 &  1  0  1  0  2  0  2  0  2  0  2  0  1  0  1\\
28 &  1.05423468693972 &  8 &  1  0  0  0 -1  0  0  0  1  0  1  0  0  0 -2\\
28 &  1.05431255509763 &  6 &  1  0  1  0  1  0  1  0  0  0 -1  0 -2  0 -3\\
28 &  1.05616339145825 &  6 &  1  0  1  0  0  0  0  0  0  0 -1  0 -1  0 -1\\
28 &  1.05637230762463 &  6 &  1  1  1  0  0 -1 -1 -1  0  0  0  0  1  0  0 P\\
28 &  1.05798761666627 &  8 &  1  0  2  0  2  0  1  0  0  0  0  0  2  0  3\\
28 &  1.05910355609770 &  6 &  1  0 -1  0  0  1  0 -1  1  0  0 -1 -1  1  1 P\\
%"28   1.05928610901145   8   1  0  1  0  0  0  0  0  1  0  2  0  1  0  -1"
%28 &  1.05975747147838 &  6 &  1  0  1  1  1  1  1  1  1  0  0  0 -1  0 -1 P\\
%%28 &  1.05993675340324 &  8 &  1  0  1  0  1  0  0  0  0  0  0  0  1  0  0\\
\\
%30 & 1.04026214469874 &  6 & 1  0  0  0  0  0  1  0  0  1  0  0  0  0  0  1\\
%32 & 1.03987206532993 &  8 & 1  0  0  0  0  0  0  0  0  0  0  0  1  0  0  0  1\\
%34 & 1.04961810533324 &  8 & 1  0  1  1  0  1  0  0  0  0  0  0  1  0  1  1  0  1\\
30 &  1.04026214469874 &  6 &  1  0  0  0  0  0  1  0  0  1  0  0  0  0  0  1\\
30 &  1.04248694101431 &  6 &  1  0  0  1  0  0  0  0  0  0  0  0  0  0  0 -1\\
30 &  1.04528115508851 &  6 &  1  0  0  1  0  0  0  0  0 -1  0  0  0  0  0  0\\
30 &  1.04978612425248 &  6 &  1  0  1  1  1  2  1  3  2  3  3  3  4  3  4  3 PM\\
30 &  1.05283588953315 &  6 &  1  0  0  0  0  0  0  0  0  0  0  0  1  0  0  1\\
30 &  1.05374090226554 &  6 &  1  0  0  0  0  0  0  0  0  1  0  0  0  0  0 -1\\
30 &  1.05561042764362 &  3 &  1  0  0  1  0  0  0  0  0 -1  0  0 -1  0  0 -1\\
30 &  1.05736311561234 & 10 &  1  0  0  0  0  2  0  0  0  0  2  0  0  0  0  1\\
30 &  1.05737367134034 &  6 &  1  1  0  0  0 -1 -1  0  0  0  1  1  0  0  0 -1 PM\\
30 &  1.05785144758134 & 15 &  1  0  0  0  0  1  0  0  0  0 -1  0  0  0  0 -3\\
30 &  1.05836091217175 &  6 &  1  1  0  0  0  0  0  0  1  0 -1  0  0  0  1  1 P\\
30 &  1.05867431886766 &  6 &  1  0  1  0  0  0 -1  0 -1  1  0  1  1  0  1 -1 P\\
\\
32 &  1.03987206532993 &  8 &  1  0  0  0  0  0  0  0  0  0  0  0  1  0  0  0  1\\
32 &  1.04014410776822 &  8 &  1  0  0  0  1  0  0  0  1  0  0  0  0  0  0  0 -1\\
32 &  1.04196421067126 &  8 &  1  0  1  0  1  0  1  0  0  0  0  0 -1  0 -2  0 -1\\
32 &  1.04307410718581 & 16 &  1  0  0  0  1  0  0  0  0  0  0  0  1  0  0  0  3\\
32 &  1.04429273277234 &  8 &  1  0  0  0 -1  0  1  0  0  0 -2  0  1  0  1  0 -1\\
32 &  1.04472633240294 &  8 &  1  0  0  0 -1  0  0  0  0  0  1  0  1  0 -1  0 -1\\
32 &  1.04566661570176 &  8 &  1  0  1  0  1  0  2  0  2  0  2  0  2  0  2  0  3\\
32 &  1.04614479501702 &  8 &  1  0  1  0  0  0  0  0  0  0  0  0  1  0  0  0 -1\\
32 &  1.04891831290646 &  8 &  1  0  0  0  0  0  1  0 -1  0  0  0  0  0 -1  0  1\\
32 &  1.04989575593276 &  8 &  1  0  0  0  2  0  0  0  2  0  0  0  1  0  0  0  1\\
32 &  1.05025977911128 &  8 &  1  0  1  0  1  0  1  0  0  0  1  0  1  0  0  0  1\\
32 &  1.05069363916743 &  8 &  1  0  0  0  0  0  1  0  0  0  0  0  1  0  0  0 -1\\
32 &  1.05077230417714 &  8 &  1  0 -1  1  0 -1  1 -1 -1  1 -1  0  1 -1  1  1 -1 P\\
%32 &  1.05088860181351 &  8 &  1  0  0  0  -1  0  0  0  1  0  0  0  0  0  0  0  1"
%"32   1.05183491235739   6   1  0  -1  1  0  -2  1  1  -2  1  2  -2  0  2  -2  -1  3"
%32 &  1.05188943173424 &  4 &  1  0  1  0  0  0 -1  0 -1  0  0  0  1  0  1  0  1\\
%32 &  1.05338808717516 &  6 &  1  0  1  0  0  0 -1  0  0  0  0  0  0  0 -1  0 -1\\
\noalign{\smallskip}\hline
\end{tabular}
\end{table}

\begin{table}[!htbp]
\caption{Irreducible, reciprocal, integer polynomials with even degree at most 34 having small house. %This list is only known to be complete through degree 20. If $d>22$ only polynomials of height one are completely investigated.
} % title of Table $\house(\alpha)$
\label{table:small4} % is used to refer this table in the text
%\centering % used for centering table
\begin{tabular}{rlrllll} % centered columns (4 columns)
\hline\noalign{\smallskip}
$d$ &	House & Out & Coefficients \\ [0.5ex] % inserts table
%heading
\noalign{\smallskip}\hline\noalign{\smallskip}
\\
34 &  1.04961810533324 &  8 &  1  0  1  1  0  1  0  0  0  0  0  0  1  0  1  1  0  1\\
34 &  1.05022062041836 &  7 &  1  1  1  1  0 -1 -2 -2 -2 -1  1  2  3  2  1 -1 -3 -3 M\\
34 &  1.05071690069432 &  8 &  1  1  0 -1  0  1  1  0  0  0  0  0  1  1  0 -1  0  0\\
34 &  1.05082250196013 &  8 &  1  1  0  0  0  0  1  1  0  0  0  0  1  0 -1  0  1  1\\
34 &  1.05105473087034 &  6 &  1  0  1  0  0  0  0  0  1  1  1  1  0  1  0  1  1  1\\
34 &  1.05115446958173 &  8 &  1  0  0  1  0 -1  1  0 -1  0  1  0  0  1  1 -1  0  1\\
34 &  1.05136643237339 &  6 &  1  1  1  0  0  0  1  0  0 -1  0  0  1  0  0 -1  0 -1\\
34 &  1.05182663296743 &  8 &  1  0 -1  0  1  1 -1 -1  1  1  0  0  0  0  1  1  0 -1\\
34 &  1.05221475176357 &  7 &  1  1  0  0  0 -1 -1  0  0  0  1  1  0  0  0 -1 -1 -1 M\\
34 &  1.05372780022456 &  8 &  1  0  0  1  0  0  0  0  0 -1  0  0 -1  1  1  0  1  1\\
34 &  1.05394569820733 &  8 &  1  0  1  0  1  1  1  1  1  1  1  1  1  1  1  0  1  0\\
34 &  1.05406220416025 &  8 &  1  1  1  1  0  0  0 -1  0  1  0  1  0 -1  1  0  0  1\\
\noalign{\smallskip}\hline
\end{tabular}
\end{table}


\clearpage
\bibliographystyle{amsplain}

\begin{thebibliography}{33}

\bibitem{A} F. Amoroso. \textit{f-transfinite diameter and number-theoretic applications}, Ann. Inst. Fourier,
Grenoble, 43 (1993), 1179--1198.
\bibitem{D} A.Dubickas. \textit{On a conjecture of A. Schinzel and H. Zassenhaus}, Acta Arith., 63 (1993), no.
1, 15--20.
\bibitem{B} D. W. Boyd, \textit{The maximal modulus of an algebraic integer}, Math. Comp. 45 (1985) 243--249. %MR790657 (87c:11097)
\bibitem{F} V. Flammang. \textit{Trace of totally positive algebraic integers and integer transfinite diameter},
Math. Comp. 78 (2009), no. 266, 1119--1125.
%\bibitem{FR} V. Flammang, G. Rhin. On the absolute Mahler measure of polynomials having all zeros
%in a sector III, to appear in Math. Comp.
\bibitem{K} L. Kronecker. \textit{Zwei S\"{a}tze \"{u}ber Gleichugen mit ganzzahligen Koeffizienten}, J. reine angew.
Math. 53 (1857), 173--175.
\bibitem{L1} M. Langevin. \textit{M\'{e}thode de Fekete-Szeg\"{o} et probl\`{e}me de Lehmer}, C. R. Acad. Sci. Paris, S\'{e}rie
1 Math. 301 (1985), no. 10, 463--466.
\bibitem{L2} M. Langevin. \textit{Minorations de la maison et de la mesure de Mahler de certains entiers
alg\'{e}briques}, C. R. Acad. Sci. Paris, 303, S\'{e}rie I, (1986), no. 12, 523--526.
\bibitem{M} E. M. Matveev. \textit{On the cardinality of algebraic integers}, Math. Notes 49 (1991), 437--438.
\bibitem{Mi} M. Mignotte. \textit{Sur un th\'{e}or\`{e}me de M. Langevin}, Acta Arith. 54 (1989), 81--86.
%\bibitem{Pari} C. Batut, K. Belabas, D. Bernardi, H. Cohen and M. Olivier, GP-Pari version 2.5.0, (
%2011).
\bibitem{RS} G. Rhin, C.J. Smyth. \textit{On the absolute Mahler measure of polynomials having all zeros in
a sector}, Math. Comp. 64 (209) (1995), 295--304.
\bibitem{RW1} G. Rhin, Q. Wu. \textit{On the absolute Mahler measure of polynomials having all zeros in a
sector II}, Math. Comp. 74 (249) (2004), 383-388.
\bibitem{RW2} G. Rhin, Q. Wu. \textit{On the smallest value of the maximal modulus of an algebraic integer},
Math. Comp. 76 (258) (2007), 1025--1038.
\bibitem{Ro} Robinson, Raphael M. \textit{Intervals containing infinitely many sets of conjugate algebraic integers},
Studies in mathematical analysis and related topics, Stanford Univ. Press, Stanford,
Calif. (1962), 305--315.
\bibitem{SZ} A. Schinzel, H. Zassenhaus. \textit{A refinement of two theorems of Kronecker}, Michigan Math.
J. 12 (1965), 81--85.
\bibitem{S1} C. J. Smyth. \textit{On the product of the conjugates outside the unit circle of an algebraic integer},
Bull. London Math. Soc, 3 (1971), 169--175.
\bibitem{S2} C.J. Smyth. \textit{The mean value of totally real algebraic numbers}, Math. Comp. 42 (1984),
663--681.
\bibitem{Sta} Stankov, D.: On the distribution modulo 1 of the sum of powers of a Salem number.  Comptes Rendus - Mathematique, 354(6), 569--576 (2016)

\bibitem{V} P. Voutier. \textit{An effective lower bound for the height of algebraic numbers}, Acta Arith. 74
(1996), no. 1, 81--95.
\bibitem{Wu} Q. Wu. \textit{On the linear independence measure of logarithms of rational numbers}, Math.
Comp. 72 (2003), 901--911.

\end{thebibliography}

\end{document}